\documentclass[journal]{IEEEtran}

\usepackage{array}
\usepackage{amsmath}%
\usepackage{cite}
\usepackage{amsfonts}%
\usepackage{amssymb}%
\usepackage{graphicx}
\usepackage{authblk}
\usepackage[margin=2cm]{geometry}
\usepackage{subfig}%
\usepackage{color}
\usepackage{dsfont}
\usepackage{comment}

\newtheorem{theorem}{Theorem}

\newtheorem{lemma}{Lemma}

\newtheorem{proposition}{Proposition}

\newenvironment{proof}[1][Proof]{\textbf{#1.} }{\ \rule{0.5em}{0.5em}}

\newcommand{\cN}{\mathcal{N}}

\clearpage
\begin{document}
\thispagestyle{empty}

\title{ 
Optimal Push and Pull-Based Edge Caching\\ For Dynamic Content

}

\author{Bahman Abolhassani, John Tadrous, Atilla Eryilmaz, Serdar Yüksel\\

\thanks{ Manuscript received February 20, 2023; revised October 11, 2023, and accepted December 14, 2023; approved by IEEE/ACM TRANSACTIONS ON NETWORKING Editor Nikhil Karamchandani. Date of publication --, 2024; date of current version January 4, 2024. This work is supported in part by the NSF grants: NSF AI Institute (AI-EDGE) 2112471, CNS-NeTS-2106679,  CNS-NeTS-2007231; and the ONR Grant N00014-19-1-2621; This research was also supported in part by the Natural Sciences and Engineering Research Council (NSERC) of Canada.

B.  Abolhassani is with the Department of Electrical and Computer Engineering, Virginia Tech, Blacksburg, VA 99202 (e-mail:abolhassani@vt.edu).

J. Tadrous is with the Department of Electrical and Computer Engineering, Gonzaga University, Spokane, WA 99202 (e-mail:tadrous@gonzaga.edu).

A. Eryilmaz is with the  Department of  Electrical and  Computer  Engineering, The Ohio State University,  Columbus,  OH  43210  USA  (e-mail:eryilmaz.2@osu.edu).

S Yüksel is with the Department of Mathematics and Statistics, Queen's University,  Kingston, Ontario  (e-mail:yuksel@queensu.ca).
}
}

\maketitle

\thispagestyle{empty}
\pagestyle{empty}

\begin{abstract}

We introduce a framework and optimal `fresh' caching for a content distribution network (CDN) comprising a front-end local cache and a back-end database. The data content is dynamically updated at a back-end database and end-users are interested in the most-recent version of that content. We formulate the average cost minimization problem that captures the system's cost due to the service of aging content as well as the regular cache update cost. We consider the cost minimization problem from two individual perspectives based on the available information to either side of the CDN: the back-end database perspective and the front-end local cache perspective. For the back-end database, the instantaneous version of content is observable but the exact demand is not. Caching decisions made by the back-end database are termed `push-based caching'. For the front-end local cache, the age of content version in the cache is not observable, yet the instantaneous demand is. Caching decisions made by the front-end local cache are termed `pull-based caching'. Our investigations reveal which type of information, updates, or demand dynamic, is of higher value towards achieving the minimum cost based on other network parameters including content popularity, update rate, and demand intensity.

\end{abstract}


\section{Introduction}
\label{sec:Intro}
The rapid increase in the number of ubiquitous wireless devices has resulted in rapidly escalating levels of data traffic over cellular networks. This surging data demand is depleting the limited spectrum resources for wireless transmission, especially over the wireless connection between the back-end databases and the end-users. With the advent of 5G networks, caching at the wireless edge has been used to accelerate the content download speed and improve the performance of wireless networks \cite{liu2016caching, yao2019mobile, paschos2018role, zhang2018cooperative, abolhassani2019wireless, abolhassani2020delay, jamali2018statistical, pouryousef2023quantum}. However,  with the emergence of new services and application scenarios, such as video on demand, augmented reality, social networking, and online gaming, which produce dynamically changing data over time, the old caching paradigms have become almost inefficient. Content freshness is becoming increasingly significant due to the fast growth of the number of mobile devices and the dramatic rise in real-time applications (e.g., video on demand, live broadcast, etc.) \cite{zhang2019price}. Thus, it is no longer sufficient to serve the content to end-users, as users favor the recent version of data over the older versions. Consequently, contents usually have the greatest value when they are fresh \cite{shapiro1998information}. The service of most fresh content from the back-end database to end-users will considerably congest the network and ultimately impair the quality of experience (QoE). Meanwhile, such trends are accelerating with the growth of cellular traffic in the network. Hence, there is a growing need to develop new caching strategies that account for the refresh characteristics and aging costs of content for efficient dynamic content distribution. 

Numerous works study the dynamic content delivery in caching systems such as \cite{najm2016age, candan2001enabling, shi2002workload, li2003freshness, wu2019dynamic, kam2017information, xu2020proactive, gao2020design,  mehrizi2019popularity, kumar2020optimized, zhong2018two, zhang2020reinforcement, masood2020learning, azimi2018online, zhang2019learning, abad2019dynamic} and effective strategies have been proposed.
In particular, authors in \cite{zhong2018two} propose two metrics to measure the cached content freshness: age of synchronization (AoS) and age of information (AoI). Most existing research regarding the freshness of the local cache focus on the AoI metric and often the objective is either to minimize the average AoI or to minimize the cache miss rate. For example, authors in \cite{kam2017information} propose a dynamic model based on AoI in which the rate of requests depends on the popularity and the freshness of information to minimize the number of missed requests.

While AoI is a meaningful metric for measuring the freshness of content in some systems \cite{maatouk2020age, chen2020age, bastopcu2019age, yang2020age}, there are many real-world scenarios where content does not lose its value simply because time has passed since it was put into the cache. These types of dynamic content include news and social network updates where the users prefer to have the freshest version but so long as there is no new update, that content is considered to be the freshest version. In this work, we use a new freshness metric called \emph{Age-of-Version} (AoV) which counts the integer difference between the versions at the database and the local cache \cite{abolhassani2021fresh, abolhassani2020achieving,abolhassani2021single,abolhassani2021optimal,abolhassani2022fresh,abolhassani2023optimal}. We also introduce a new cost function for dynamic content caching which captures both the cost due to the missed event and the cost due to content freshness \cite{wessels2001web} which grows with the AoV metric. Our model extends the traditional caching paradigm to allow for varying \textit{generation dynamics} of content and calls for new designs that incorporate these dynamics into its decisions.

In particular, despite the traditional caching paradigms, for content with varying generation dynamics, it is essential to regularly update the cached content to prevent the cache from being obsolete due to old age. This requires that either the back-end database or the front-edge local cache, routinely check for updates and refresh the local cache if needed. There are two main dynamics at work here. On the back-end side, content is being refreshed at the source, and on the front-end side, requests arrive at the local cache according to a popularity profile. Depending on whether the edge-cache or the back-end database takes control of the cache updates, there will be two different caching strategies that have access to different live information. 

If the responsibility of the cache updates is put on the back-end database, it has access to the most current version of the content and it also knows the version of the cached content at no cost, i.e., does not incur additional cost to keep track of the age of version at the local cache. The back-end database can thus determine the exact age of the cache at any given time and make a decision whether to update the cache or not based on that information. On the other hand, if the front-end edge-cache is responsible for the cache updates,  it will not know the exact age of the cached content. However, the edge-cache knows the exact times of the arrival requests and can choose whether to update the cache or not based on that information. This creates two distinct dynamic caching paradigms that we will discuss in this paper.

There exist works on push-based and pull-based caching for static content such as \cite{hara2002cooperative, gwertzman1995case, gwertzman1997analysis, zhu2004stochastic, kambalakatta2004profile, lee2005performance}.  For example, authors in \cite{hara2002cooperative} consider a system where the server repeatedly broadcasts data to clients through a broadband channel and propose caching strategies where clients cooperatively cache broadcast data items. Also, \cite{gwertzman1995case} proposes geographical push-caching to benefit the server’s global knowledge of the situation. In another work, authors in \cite{kambalakatta2004profile} propose profile-based caching to determine items to be prefetched and cached depending on a user group and context. However, these works fail to perform well once the static assumption is taken away. In our earlier work \cite{abolhassani2021fresh}, we explored pull-based caching for dynamic content without providing proof of its optimality. There, we showed that when the cache size is limited, only items that are popular enough will be considered for caching. In this work, we not only consider push-based caching alongside pull-based caching, but we also provide the optimality proof for both of the proposed policies. Our focus in this work is on the best approach to manage the cache updates once an item is put in the cache. We will initially study the unlimited buffer case which gives us the foundation to add the buffer constraint.

In this work, we characterize the problem of dynamic content caching from two stances, one from the perspective of the front-end edge-cache and the other from the perspective of the back-end database. More specifically, we present the optimal caching policy for each paradigm and provide detailed proof of their optimality. Comparing the two optimal paradigms for a single item reveals that depending on the refresh rate and popularity of that item, one of the paradigms will outperform the other. We explicitly characterize the conditions on the popularity and refresh rate of the item where each paradigm outperforms the other. This opens the door for a compromised caching paradigm where contents will be divided into two groups based on their popularity and refresh rates. For one group, the edge-cache will take the responsibility for cache updates, and for the other group, the back-end database will be in charge of the cache updates. The result is a caching scheme that outperforms each paradigm if applied exclusively. 

For each specific item, by intelligently choosing whether the edge-cache or the back-end database is in charge of the cache updates and then applying the optimal policy from the perspective of the edge-cache and back-end database to each group, the average caching cost is minimized. We aim to reveal the potential gains by comparing the proposed paradigm to an unachievable genie-aided paradigm where it knows both the exact age of the cache and the exact time of the request arrivals. Our contributions, along with the organization of the paper, are as follows.

\begin{list}{\labelitemi}{\leftmargin=1em}
\item In Section~\ref{sec:Sys_Model}, we present a tractable caching model for serving dynamic content to end-users from a back-end source and formulate the general problem.
\item In Section~\ref{sec:PushBased}, we consider the problem of dynamic content caching from the perspective of the back-end database where the back-end database, which has access to the update information, is in charge of the cache updates. We characterize the structure of the optimal Push Based Edge Caching Policy, and provide detailed proof of its optimality. The optimal policy for the push-based caching paradigm decides to update the cached content once its age reaches a threshold.

\item In Section~\ref{sec:PullBased}, we study the caching problem from the perspective of the front-end edge-cache where the local cache, which has access to demand dynamics, is in charge of the cache updates. We characterize the optimal Pull Based Edge Caching Policy and provide detailed proof of its optimality. The optimal policy for the pull-based caching paradigm decides to update the cache when the elapsed time since the last update is large enough and once there is a request.

\item In Section~\ref{sec:proposed}, comparing the two paradigms for single-item caching reveals scenarios in which each paradigm outperforms the other depending on the popularity and refresh rate of the item. Thus we present a new paradigm where items are divided into two disjoint groups, where the push-based paradigm is applied to one group and the pull-based paradigm is applied to the other. The result is a caching scheme that outperforms both paradigms if utilized exclusively. We then generalize the proposed combined policy to incorporate the limited buffer case. Moreover, we compare the proposed caching scheme with an unachievable genie-aided paradigm to reveal the potential gains that can be achieved by combining the two optimal caching paradigms. Finally, we conclude the work in Section~\ref{sec:Conclusion}.
\end{list}

\section{System Model}
\label{sec:Sys_Model}
We consider the generic hierarchical setting depicted in Fig.~\ref{fig:system,model}, whereby: a local edge-cache serves a user population that generates content requests according to some popularity distribution. The local cache is connected to a back-end database where different contents receive consistent updates. Here is a detailed description of our system's components and associated dynamics.

\noindent \textbf{Demand Dynamics:} We assume that a set $\cN$ of $N$ unit-size data items with dynamically changing content is being served to a user population by the hierarchical caching system in Fig.~\ref{fig:system,model}. Requests arrive at the local cache according to a Poisson process\footnote{Accordingly, we assume that the system evolves in continuous time.} with rate $\beta\geq0$, which captures the request intensity of the user population. An incoming request targets data item $n \in \cN$ with probability $p_n$. Accordingly, the probability distribution $\mathbf{p}=(p_n)_{n=1}^N$ captures the popularity profile of the data items. We consider a binary random variable $r_n(t) \in  \{0,1\}$ to capture the event of item $n$ being requested at time $t$, i.e., $r_n(t) =1$ if there is a request at time $t$, and $r_n(t) =0$ in the case of no request.

\noindent \textbf{Generation Dynamics:} At the database, each data item may receive updates to replace its previous content. We assume that data item $n$ receives updates according to a Poisson process with the rate $\lambda_n\geq 0$. Note that $\lambda_n=0$ encapsulates the traditional case of \textit{static} content that never receives updates. We denote vector $\boldsymbol{\lambda}=(\lambda_n)_{n=1}^N$ as the collection of update rates for the database. We consider a binary random variable $b_n(t) \in  \{0,1\}$ to capture the event of item $n$ receiving an update at the back-end database at time $t$, i.e., $b_n(t) =1$ if there is an update, and $b_n(t) =0$ in the case of no update.

\noindent \textbf{Age Dynamics:} Since the data items are subject to updates at the back-end database, the same items in the local cache may be \textit{older versions} of the content. To measure the freshness of local content, we define the \emph{age} $\Delta_n(t)\in\{0,1,\ldots\}$ at time $t$ of a cached content for item $n$ as the number of updates that the locally available item $n$ has received in the back-end database since it has been most recently cached.  We name this freshness metric as the \emph{Age-of-Version} (AoV) since it counts the integer difference between the versions at the database and the local cache. Now that we have the dynamics defined, we can introduce the key operational and performance costs associated with our caching system. 

\noindent \textbf{Fetching and Aging Costs:} On the operational side, we denote the cost of fetching an item from the back-end database to the local cache by $c_f>0.$ On the performance side, we assume that serving an item $n$ from the local cache with age $\Delta_n(t)$ incurs a \textit{freshness/age} cost of $c_a \times \Delta_n(t)$ for some $c_a \geq 0$, which grows linearly\footnote{While this linearity assumption is meaningful as a first-order approximation to the aging cost and facilitates simpler expressions in the analysis, it can also be generalized to convex forms to extend this basic framework.} with the AoV metric. This aging cost measures the growing discontent of the user for receiving an older version of the content she/he demands. For each data item $n$, we define the binary action $u_n(t) \in \{0,1\}$, to capture the decision of fetching the most recent version of content $n$ to the edge-cache at time $t$, i.e., $u_n(t)=1$. On the other hand, $u_n(t)=0$ captures the case when the request will be served from the edge-cache incurring the freshness cost.

\noindent \textbf{Problem Statement:} 
Our broad objective in this work is to develop efficient caching strategies for the above setting that optimally balance the tradeoff between the cost of frequently updating local content and the cost of providing aged content to the users. In particular, we are interested in finding a policy that minimizes the long-term average cost of the system by deciding when to update the cache and when to serve the requests from the edge-cache. We can express this goal generically as
\begin{equation}
\label{eq:costmin}
\begin{split}
    &\min _{\boldsymbol{u}} \lim _{\mathrm{T} \rightarrow \infty} \frac{1}{T} \mathbb{E}_{x}^{u}\left[\int_{o}^{\mathrm{T}} \sum_{n=1}^N C_{n}^{\pi}\left(x_{n}(t),u_n(t)\right) dt\right]\\
    & s.t. \quad  u_n(t) \in \{0,1\}, \quad \forall n \in \cN , t\geq 0,
\end{split}
\end{equation}
where $u_n(t)\in \{0,1\}$ and $x_n(t)$ are the action and the state of the system at time $t$ for data item $n$, respectively. Thus, $C_{n}^{\pi}\left(x_n(t),u_n(t)\right)$ represents the cost of the system for data item $n$ at time $t$ when the system is at the state $x_n(t)$ and the action $u_n(t)$ is taken under some paradigm $\pi$. The operator $\mathbb{E}_{x}^{u}$ captures the expected value over the system state $x$ for a given action $u$.

Due to the dynamic nature of the content, if cached items are not updated frequently, they will become obsolete and hence increase the average cost. Therefore, an update controller is needed to routinely check for the updates of the cached content and decide when to fetch the freshest version to the local cache.
Such updates can be performed either by the local edge-cache or by the back-end database, resulting in completely different cache update processes. This is due to the fact that the back-end database and the local edge-cache have access to different live information since they are located at the opposite end of the caching system. The former has access to update information and the latter knows the demand information. Subsequently, we consider the average cost minimization from the perspective of the back-end database and the front-end local cache, individually. 

\noindent \textbf{Push-Based Caching:} 
Here, the back-end database decides on when to fetch the freshest version of every content to the local cache, incurring a cost of $c_f$. Moreover, the back-end database knows the exact age of the cached content but has no access to the specific times of request arrivals to the local cache. Therefore, the optimal policy for the back-end database driven paradigm will be to decide when to push the item to the cache given the exact age of the cached content.

\noindent \textbf{Pull-Based Caching:} 
Here, the local edge-cache is responsible for deciding when to request the most recent version of locally cached item $n$ from the back-end database, at a fetching cost of $c_f$. Furthermore, the local cache knows the exact times of the request arrivals but it is not aware of the exact age of the cached content. Therefore, the optimal policy for the edge-cache driven paradigm will be to decide when to pull the item from the database given that it knows the exact times of the request arrivals.

Next, we characterize the optimal caching policies for both paradigms and compare their performance to identify the situations where each is superior to the other. Finally, by combining the two paradigms, we propose a caching scheme that achieves the minimum average cost of two paradigms for each item $n$, outperforming both of the strategies if deployed singly. 

\begin{figure}[t]
\centering
\includegraphics[width=0.5\textwidth]{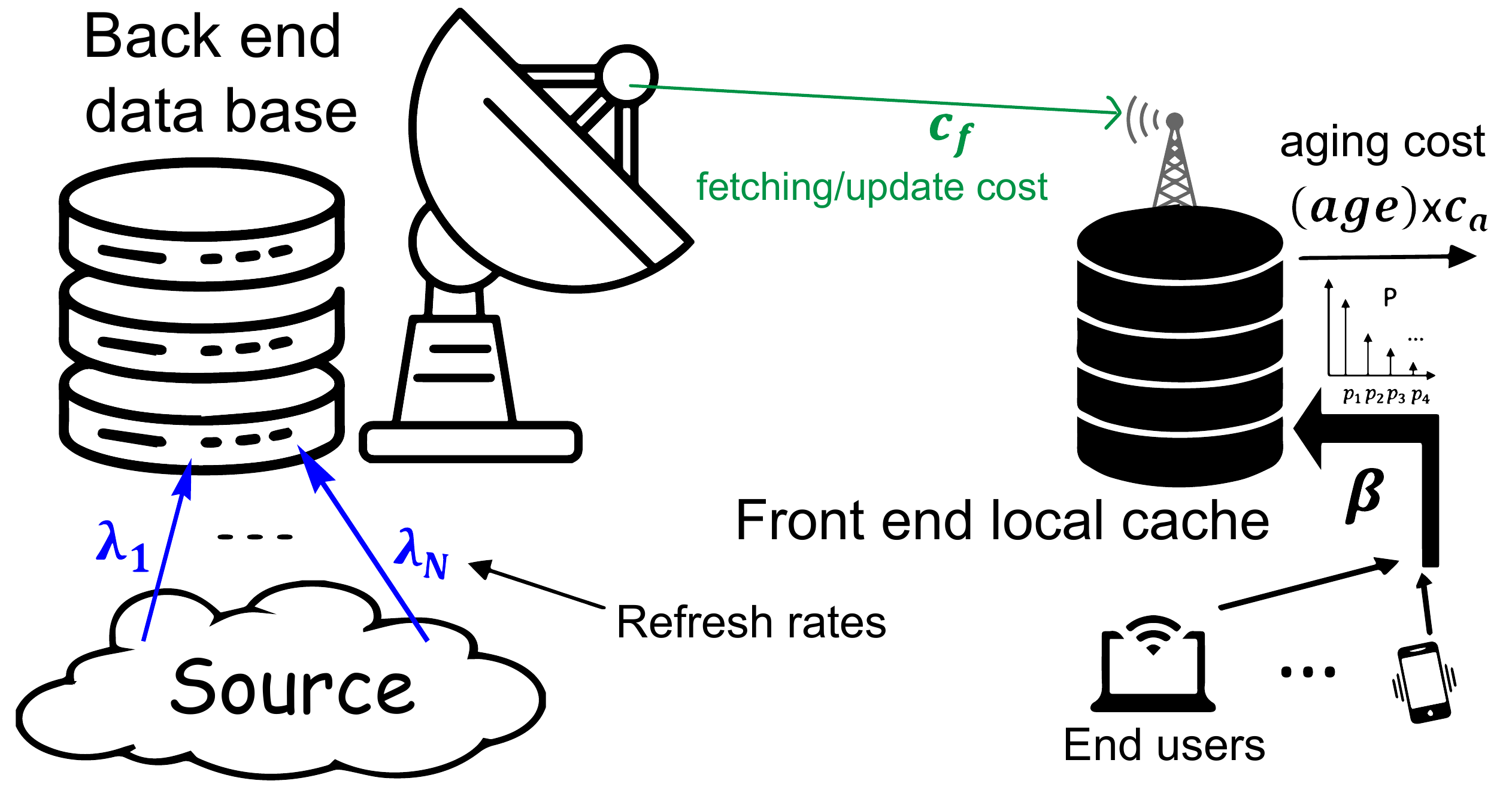}
\caption{Setting of \textit{Fresh Caching for Dynamic Content}}
\label{fig:system,model}
\end{figure}

\section{Push-Based Edge Caching} 
\label{sec:PushBased}
In this section, we focus on the general cost minimization problem expressed in Equation \eqref{eq:costmin} when the back-end database is responsible for updating the edge-cache. The database has instantaneous knowledge of the exact version of each content at any given time. However, the database is not aware of the exact time of the request arrivals to the edge-cache, instead, it only knows the rate of the requests $\beta$ and the popularity profile $\textbf{p}$. Accordingly, the system state accessible for the back-end database driven caching paradigm is $x_n(t)=(\Delta_n(t)), \forall n\in \cN, t\geq 0$. Therefore, at any time $t$, the expected cost of the back-end database driven caching paradigm for data item $n$ can be expressed as:
\begin{equation} 
    \mathbb{E}_{x}^{u}\left[C^B_n(\Delta_n(t),u_n(t))\right]=u_n(t) c_f + \beta p_n  c_a \Delta_n(t),
\end{equation}
where the first term captures the cost of the cache update if the back-end database decides to push the freshest version to the edge-cache, i.e., $u_n(t)=1$, and the second term captures the freshness cost due to serving a potentially aged item from the cache. Thus, the cost minimization problem \eqref{eq:costmin} for the back-end database driven caching paradigm can be expressed as:
\begin{equation} 
\label{eq:costminpush}
\begin{split}
    &\min _{\boldsymbol{u}} \lim _{\mathrm{T} \rightarrow \infty} \frac{1}{T} \int_{o}^{\mathrm{T}} \sum_{n=1}^N \left[ u_n(t) c_f + \beta p_n  c_a  \Delta_n(t)\right]dt\\
    & s.t. \quad  u_n(t) \in \{0,1\}, \quad \forall n \in \cN , t\geq 0,\\
    & \hspace{.1in}  \Delta_n(t+dt) =\Delta_n(t) (1-u_n(t))+b_n(t), \forall n \in \cN , t\geq 0.
\end{split}
\end{equation}
Next, we give the optimal back-end database driven caching policy that solves problem \eqref{eq:costminpush} and call it the optimal policy for the Push-Based Caching paradigm. 
\begin{theorem} \label{thm:costpush}
Policy $\boldsymbol{u}^*$ that solves \eqref{eq:costminpush} is given by:
\begin{equation}
    u^{*}_n(\Delta_n)=\left\{\begin{array}{cc}
1 & \Delta_n>\eta_n^* \\
0 & \Delta_n \leq \eta_n^*,
\end{array}\right.
\end{equation}
where $\eta_n^*=\underset{m \in \{\eta_n,\eta_n+1\} }{\operatorname{argmin}} \hspace{.04in} C^{B}_n(m)$ such that $\eta_n=\left\lfloor\sqrt{2 \lambda_n c_{f} / \beta p_n c_{a}}\right\rfloor$ and $C^{B}_n(m)=\frac{1}{2}\beta p_n c_a (m-1)+\frac{\lambda_n c_f}{m}$. Thus, $C^{B^*}=\sum_{n=1}^{N}  C^{B}_n(\eta_n^*)$ is the optimal cost for the back-end database driven caching paradigm. Under the assumption that $\sqrt{2 \lambda_n c_{f} / \beta p_n c_{a}}$ is an integer, the corresponding optimal average cost is given by:
\begin{equation} 
    C^{B^*}=\sum_{n=1}^N \left[ \sqrt{2 \lambda_n \beta p_n c_{a} c_{f}}-\frac{1}{2} \beta p_n c_{a}\right].
\end{equation}
\end{theorem}
\begin{proof} 
We start the proof by first revealing that the optimal paradigm has a threshold based structure and then fully characterize the optimal threshold. For the full detailed proof, please refer to Appendix A. 
\end{proof}

According to Theorem \ref{thm:costpush}, the optimal push-based paradigm utilizes the knowledge of the exact age of the cached content to decide when to update the cache. Consequently, it decides to push the most fresh version to the cache, immediately after the age of the cached content reaches a threshold. Moreover, As Theorem \ref{thm:costpush} shows, the optimal push-based paradigm is a deterministic paradigm. In other words, independent of time, back-end database can choose whether to update the cache or not based on the current age of the cached content.

In the next section, we consider the average cost minimization problem \eqref{eq:costmin} from the perspective of the edge-cache. i.e.,  the local cache controls the cache updates.

\section{Pull-Based Edge Caching} 
\label{sec:PullBased}
In this section, we focus on the general problem \eqref{eq:costmin} when the front-end local cache is responsible for updating the cache. Since the requests arrive to the local cache, the edge-cache observes the exact times of the arriving requests. However, the local cache is not aware of the exact age of the cached items and it only knows the rates of the updates $\boldsymbol{\lambda}$. Moreover, define $s_n(t)=\{t-l : l= \max (t' \leq t : u_n(t' )=1)\}\geq 0, \forall n \in \cN$ to be the elapsed time since the last time item $n$ was updated in the cache. Thus, at any time $t$, $\lambda_n s_n(t)$ would be the expected age of the data item $n$ in the cache. Accordingly, the system state accessible for the edge-cache driven caching paradigm is $x_n(t)=(s_n(t), r_n(t)), \forall n\in \cN, t\geq 0$.
As a consequence, at any time $t$, the expected cost of the edge-cache driven caching paradigm for data item $n$ can be expressed as:
\begin{equation} 
\begin{split}
\mathbb{E}_{x}^{u}\left[C^E_n(s_n(t),r_n(t), u_n(t))\right]&=u_n(t) c_f \\&+ r_n(t) s_n(t) \lambda_n c_a (1-u_n(t)),
\end{split}
\end{equation}
where the first term captures the cost of fetching the freshest version from the database and the second term captures the freshness cost due to serving a potentially aged item from the cache. Thus, the cost minimization problem \eqref{eq:costmin} for the edge-cache driven caching paradigm can be expressed as:
\begin{equation}
\label{eq:costminpull} 
\begin{split}
     &\hspace{-.05in} \min _{\boldsymbol{u}} \lim _{\mathrm{T} \rightarrow \infty} \frac{1}{T} \int_{o}^{\mathrm{T}} \sum_{n=1}^N \left[ u_n(t) c_f + r_n(t) s_n(t) \lambda_n c_a (1-u_n(t))\right]dt\\
    &\hspace{-.05in}  s.t. \quad  s_n(t+dt)=s_n(t)(1-u_n(t))+dt, \quad \forall n \in \cN , t\geq 0,\\
    &\hspace{-.05in}  \hspace{.32in} u_n(t) \in \{0,1\}, \quad \forall n \in \cN , t\geq 0,\\
    &\hspace{-.05in}  \hspace{.32in} r_n(t) \in \{0,1\}, \quad \forall n \in \cN , t\geq 0,
\end{split}
\end{equation}
Next, we give the optimal edge-cache driven caching policy that solves problem \eqref{eq:costminpull} and call that the optimal policy for the Pull-Based Caching paradigm. 
\begin{theorem} \label{thm:costpull}
Policy $\boldsymbol{u}^*$ that solves \eqref{eq:costminpull} is given by:
\begin{equation}
    u^{*}_n(s_n,r_n)=\left\{\begin{array}{cc}
r_n & s_n>\tau_n^*, \\
0 & s_n \leq \tau_n^*,
\end{array}\right.
\end{equation}
where $\tau_n^*=\frac{1}{\beta p_{n}}\left(\sqrt{1+2 \frac{\beta p_{n} c_{f}}{c_{a} \lambda_{n}}}-1\right)$  is the optimal time threshold that the edge-cache will fetch the item from the database upon receiving a new request. Then, the corresponding optimal average cost is given by:
\begin{equation}
C^{E^*}(\boldsymbol{\lambda}, \beta, \mathbf{p})=\sum_{n=1}^N c_{a} \lambda_{n}\left(\sqrt{1+2 \frac{\beta p_{n}}{\lambda_{n}} \frac{c_{f}}{c_{a}}}-1\right).
\end{equation}

\end{theorem}
\begin{proof}
We start the proof by first revealing that the optimal policy has a threshold-based structure and then fully characterize the optimal threshold. For the full detailed proof, please refer to Appendix B. 
\end{proof}

According to Theorem \ref{thm:costpull}, the optimal pull-based policy utilizes the knowledge of the request arrivals and the time elapsed since the last update to decide when to fetch the most fresh version to the cache. Consequently, it decides to fetch the up-to-date version to the cache, the two criteria are met. First, the cached item $n$ must be in the cache for longer than a constant time $\tau_n$ without receiving any update. Secondly and more importantly, there must be a request for item $n$ to trigger a fetching and thus update the cache. In other words, cache updates are only performed when there is a request for an item that has been in the cache for longer than a set time. Moreover, as Theorem \ref{thm:costpull} shows, the optimal pull-based policy is a deterministic policy. In other words, independent of time, the edge-cache can decidedly choose whether to update the cache or not based on the current elapsed time and whether there is a request or not.

It turns out that depending on the popularity and refresh rate for each item, one of the push or pull-based caching paradigms can outperform the other. This is the idea behind the next section where we arrange the items into two groups based on their popularity and refresh rates and deploy one of these caching paradigms for each group. The result is a caching scheme that outperforms each one separately.

\section{Proposed Combined Caching} 
\label{sec:proposed}
In this section, we investigate the general problem defined in equation \eqref{eq:costmin}. We compare two caching paradigms, as described in Sections \ref{sec:PushBased} and \ref{sec:PullBased}. We demonstrate that, for each data item $n \in \cN$, there exists a threshold in the fraction $\frac{p_n}{\lambda_n}$ such that the push-based paradigm will be more effective than the pull-based paradigm. Based on this threshold, we can divide the items into two groups and apply the push-based caching paradigm for the first group and the pull-based paradigm for the second group.

In this section, we present the method for calculating the threshold, show how our proposed caching scheme outperforms previous paradigms, and compare the results with a genie-aided paradigm that has access to all system information. We first consider the scenario without buffer constraints to gain insight into the optimal combined policy. Then, we take into account buffer constraints and modify the proposed policy for the limited cache scenario.

\subsection{Caching without Buffer Constraint}
Here, we examine the scenario where the cache size is unlimited and there is no buffer constraint. The objective is to develop a combined policy that effectively integrates push-based and pull-based caching to reduce the average cost. First, we start by demonstrating the gain of the push-based compared to the pull-based caching paradigm and illustrate that for each item $n$, the gain only depends on the fraction $\frac{p_n}{\lambda_n}$. In Theorem \ref{thm:costpush}, by assuming that $\sqrt{2 \lambda_n c_{f} / \beta p_n c_{a}}$ is an integer, we can express the optimal push-based caching cost for item $n$ as:
\begin{equation}
    C^{B^*}_n=\sqrt{2 \lambda_n \beta p_n c_{a} c_{f}}-\frac{1}{2} \beta p_n c_{a}.
\end{equation}
On the other hand, according to Theorem 2, the optimal pull-based caching cost for item $n$ is given by:
\begin{equation}
    C^{E^*}_n=\lambda_n c_{a}\left(\sqrt{1+\frac{2 \beta p_n c_{f}}{\lambda_n c_{a}}}-1\right).
\end{equation}
We define the percentage cost reduction for data item $n$ as:
\begin{equation} \label{eq:reduction}
    Reduction (\%) = 100 \times \frac{C^{E^*}_n-C^{B^*}_n}{C^{B^*}_n},
\end{equation}
where it can be expressed as:
\begin{equation}
\label{eq:costreduction}
   Reduction (\%) = 100 \times \left(\frac{\sqrt{1+G F}-1}{\sqrt{G F}-F}-1\right),
\end{equation}
such that $F=\frac{\beta p }{2 \lambda}$ and $G=\frac{4c_f}{c_a}$. For the purpose of continuous optimization over $F$, the indices $n$ have been omitted in our analysis, as we focus solely on a single item.

This shows that for a given $G$ and $\beta$, the gain is only a function of the fraction $\frac{p}{\lambda}$. Moreover, it also shows that the threshold where the gain is zero only depends on the fraction $\frac{c_f}{c_a}$ and not on any other system parameter. The percentage cost reduction given in Equation \eqref{eq:costreduction} is depicted in Fig. \ref{fig:costreduction} for different values of $F$ and $G$. As the figure shows, for a given $G$, there is a threshold where the gain becomes zero, i.e., the two paradigms yield the same average cost. Define $f^*(G)=\{2F | Reduction (\%) =0 \}$ to be the fraction of popularity to refresh rate where the gain is zero for the given $G$. As the figure shows, for any data item $n$ and a given $G$, if $\frac{p_n}{\lambda_n}<f^*(G)/\beta$, then the pull-based caching outperforms the push-based caching. On the other hand, if $\frac{p_n}{\lambda_n}>f^*(G)/\beta$, then the push-based caching outperforms the pull-based caching. 
In the next proposition, we show that such $f^*(G)$ always exists and is finite.
\begin{proposition}
\label{prop:gain_threshold}
For any given $G=\frac{4c_f}{c_a}$, the zero gain threshold $f^*(G)$ exists and satisfies the following equation:
\begin{equation}
    1<f^*(G)< 2.
\end{equation}
\end{proposition}
\begin{proof}
Setting the cost reduction given in Equation \eqref{eq:reduction} to zero, we have:
\begin{equation*}
    F^3-4(1+G)F^2+4(1+2G)F-4G=0,
\end{equation*}
where it is a cubic equation that has at least one real root. For $F=1$, the left-hand side is equal to $1$ and is positive. For $F=1/2$, the left-hand side is equal to $9/8-G$ and is negative under any practical assumption where $c_f>c_a$. Therefore, there exists a $F^*$ such that $1/2<F^*<1$ that solves the above equation. According to definition  $F=\frac{\beta p_n }{2 \lambda _n}$ and $f^*(G)=\{f=\frac{ \beta p_n}{\lambda_n} | Reduction (\%) =0 \}$, there is a $1<f^*(G)< 2$ that results in zero percentage gain.
\end{proof}

Utilizing the results of the Proposition \ref{prop:gain_threshold}, the numerical results in Fig. \ref{fig:costreduction} confirm that for any data item $n$, if $\frac{\beta p_n}{\lambda_n} <1$, pull-based caching will always outperform the push-based caching. On the other hand, if $\frac{\beta p_n}{\lambda_n} >2$, push-based caching will always outperform the other, under any given system parameters.

Based on these findings, we propose a caching scheme that utilizes this threshold to arrange items into two different groups based on their popularity and refresh rates.

\begin{figure}
\centering
\includegraphics[width=0.5\textwidth]{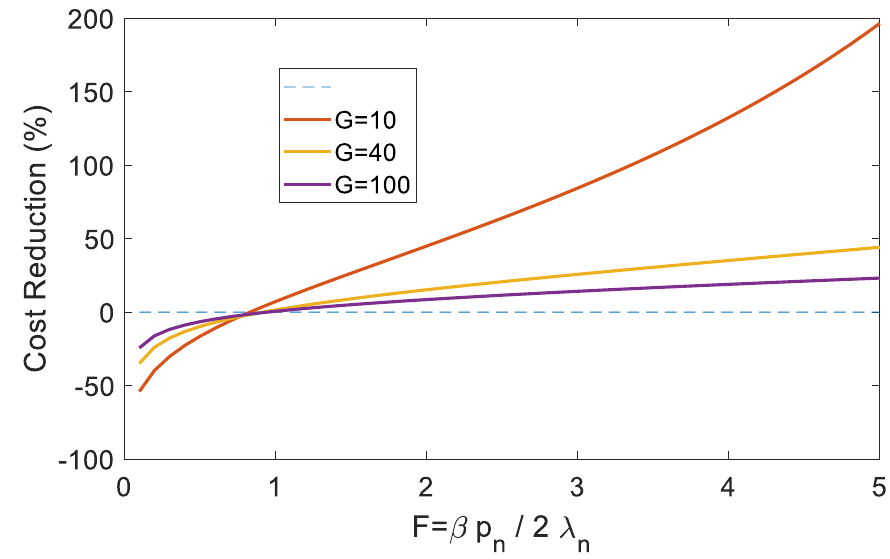}
\caption{Cost Reduction (\%) for Push-Pull based caching }
\label{fig:costreduction}
\end{figure}


\begin{theorem}
\label{thm:combinedparadigm}
In a system composed of a data set $\cN$ of $N$ items with popularity distribution $\mathbf{p}=(p_n)_{n=1}^N$ and update rates $\boldsymbol{\lambda}=(\lambda_n)_{n=1}^N$, assume without loss of generality that items are ordered such that $y_{1}^{*} \geq y_{2}^{*} \geq ... \geq y_{N}^{*}$ where $y_{n}^{*}$ is defined as $y_{n}^{*}=\frac{\beta p_n}{\lambda_n}$. Define $n^*=\max \{ n | y_n^* >f^*\}$ and arrange items in two groups where $\mathcal{G}_1=\{1,2,...,n^*\}$ and $\mathcal{G}_2=\cN/\mathcal{G}_1$. For the first group apply the optimal push-based caching policy (Theorem \ref{thm:costpush}) and let the back-end database control the cache updates. For the second group, utilize the optimal pull-based caching policy (Theorem \ref{thm:costpull}) and let the edge-cache control the updates. The result is a caching scheme with the corresponding average cost given by:
\begin{equation} \label{eq:costcombined} 
\begin{split}
    C^*= & \sum_{n=1}^{n^*} \left( \sqrt{2 \lambda_n \beta p_n c_{a} c_{f}}-\frac{1}{2} \beta p_n c_{a}\right) \\
    + & \sum_{n^*+1}^{N}
        \lambda_n c_{a}\left(\sqrt{1+\frac{2 \beta p_n c_{f}}{\lambda_n c_{a}}}-1\right).
\end{split}
\end{equation}
\end{theorem}
\begin{proof}
As we have discussed so far, for the problem of edge caching in the presence of dynamic content, the cache needs to be updated regularly. The update process can either be controlled by the front-end edge-cache or by the back-end database, resulting in pull-based and push-based caching strategies, respectively. For each scenario, we explicitly characterized the optimal policy in Theorems \ref{thm:costpush} and \ref{thm:costpull}. According to Proposition \ref{prop:gain_threshold}, there $\exists$ a threshold $f^*$ for $\frac{\beta p_n}{\lambda_n}$ that both policies will have identical performance, resulting in zero percentage gain. Moreover, if $\frac{\beta p_n}{\lambda_n}>f^*$, push-based caching performs better and if $\frac{\beta p_n}{\lambda_n}<f^*$, pull-based caching performs better than the other. Therefore, we divide items into two groups based on the fraction $\frac{\beta p_n}{\lambda_n}$ and apply the superior paradigm to each group of items. The result is a new combined scheme that outperforms each paradigm separately. Finally, according to Theorems \ref{thm:costpush} and \ref{thm:costpull}, the average cost can be expressed as in Equation \eqref{eq:costcombined}.
\end{proof}

Such caching scheme first decides which of the push or pull-based caching strategies should be deployed for each item, dividing items into two disjoint groups, and then utilizing the optimal push or pull-based edge caching for those groups of items. In other words, for each item, the new combined scheme chooses one of the push or pull-based caching paradigms that outperforms the other. Therefore, it outperforms each one if deployed exclusively. 

Next, to better demonstrate the potential of the proposed combined scheme, we consider a genie-aided paradigm based on full access to system states. Suppose there is a cache update controller unit that is fully aware of the exact time of request arrivals and also exactly knows the age of the cached items at all times, i.e., both age aware and request aware. Also, assume this information is available to the controller at all times at no extra cost. Accordingly, the system state accessible for such genie-aided paradigm is $x_n(t)=(\Delta_n(t), r_n(t)), \forall n\in \cN, t\geq 0$.
The result is a hypothetical paradigm that will reveal a lower bound on the minimum achievable average cost. In other words, there is no other paradigm that can achieve any cost less than this genie-aided paradigm. We aim to compare the proposed optimal paradigms to this impractical genie-aided paradigm.

\begin{proposition} \label{prop:genie-aided} (Hypothetical genie-aided paradigm) Consider a hypothetical genie-aided paradigm that is both age and request aware. In other words, it is not only fully aware of the exact version of the contents in the back-end database, but it also knows the exact times of the request arrivals to the edge-cache. Also, assume this information is available to the controller at all times at no extra cost. Under this assumption, the optimal policy is given by:
\begin{equation} \label{eq:combinedoptimal}
    u^{*}_n(\Delta_n,r_n)=\left\{\begin{array}{cc}
r_n & \Delta_n>\eta_n \\
0 & \Delta_n \leq \eta_n,
\end{array}\right.
\end{equation}
where $\eta_n=\underset{m \in \{0,1,2,...\} }{\operatorname{argmin}} \hspace{.04in} C_n(m)$ such that $C_n(m)=\frac{0.5\beta p_n c_am(m-1)+\lambda_n c_f}{\lambda_n/\beta p_n+m}$. Thus, $C^{*}=\sum_{n=1}^{N}  C_n(\eta_n^*)$ is the optimal cost under this hypothetical genie-aided paradigm.
\end{proposition}

\begin{proof}
Since such a paradigm is both age and request aware, the state would be $x_n(t)=(\Delta_n(t)),r_n(t))$. Under the hypothetical assumption that this information is available to the controller for free at all times, the expected cost can be expressed as:
\begin{equation} 
\begin{split}
    \mathbb{E}_{x}^{u}\left[C_n(\Delta_n(t),r_n(t),u_n(t))\right]&=u_n(t) c_f \\&+ r_n(t)  c_a \Delta_n(t)(1-u_n(t)).
    \end{split}
\end{equation}

The cost minimization problem for such a paradigm would be to solve Equation \eqref{eq:costmin} under assumptions $\Delta_n(t+dt) =\Delta_n(t) (1-u_n(t))+b_n(t)$ and $ r_n(t) \in \{0,1\}, \quad \forall n \in \cN , t\geq 0$. Utilizing a similar approach we took in Appendices A and B will yield the optimal paradigm as in \eqref{eq:combinedoptimal}. Replacing the optimal age threshold in $C_n(m)$ and summing over all items would result in the optimal average cost.
\end{proof}

As we can see in Equation \eqref{eq:combinedoptimal}, such a hypothetical genie-aided paradigm will decide to update the cache when the age is above a threshold and there is a new request. Since we assumed that age and request information are available to the controller at all times with no extra cost, the cost under such a paradigm gives a lower bound on the average cost.

We would like to highlight that such a hypothetical genie-aided policy, although conceptually intriguing, is not practically viable. This is due to the large number of factors that the base station would need to manage. In the context of this genie-based approach, the base station's responsibilities would extend beyond understanding update patterns. It would necessitate knowledge of the specific content stored within each individual edge cache, the precise ages of this content, and the exact timings of requests made to each edge cache. Effectively keeping track of all these factors becomes increasingly unfeasible, particularly when dealing with a large number of edge caches. For this reason, we exclusively utilize the genie-aided policy as a benchmark policy against which we evaluate the performance of our proposed policies.

Next, we compare the performance of proposed paradigms to the case of the hypothetical genie-aided paradigm defined in Proposition \ref{prop:genie-aided} using numerical simulations. We consider the simulation parameters to be $\beta=5$ for the average total request rate and the normalized fetching and aging costs to be $c_f=1$ and $c_a=0.1$ respectively. We assume that the database consists of $N=10^3$ items. Moreover, assume the item popularities follow a Zipf distribution with parameter $z$, i.e., $p_n=\frac{p_0}{n^z}, \forall n \in \cN$ and refresh rate is the same for all the items, i.e., $\lambda_n=\lambda, \forall n \in \cN$.

\begin{figure}[t]
\centering
\includegraphics[width=0.5\textwidth]{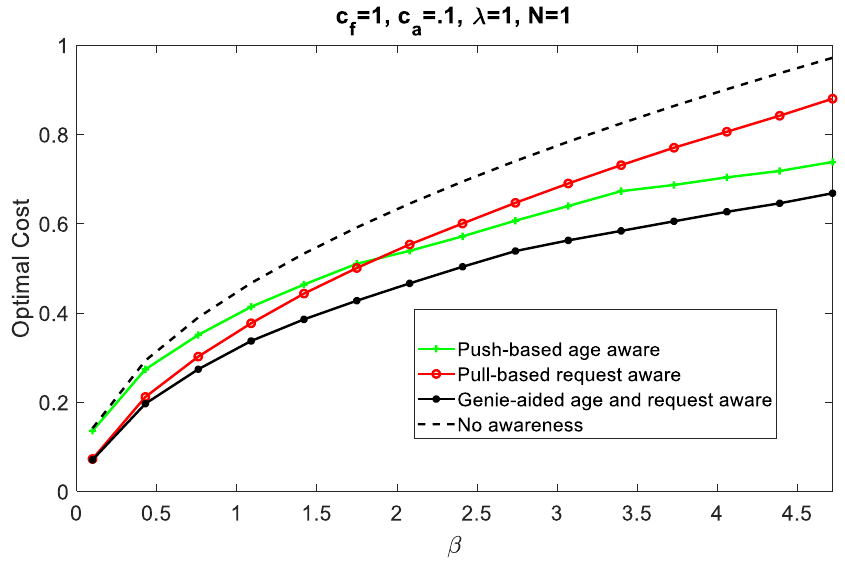}
\caption{ Performance Comparison of the Proposed paradigms for Single Item.}
\label{fig:comparison1}
\end{figure}
Fig.~\ref{fig:comparison1} shows the performance of paradigms for the single item scenario as a function of request arrival rate $\beta$. According to the figure, the pull-based caching paradigm that is only request aware outperforms the push-based caching paradigm (only age aware) when $\beta$ is small. However, as $\beta$ increases, the tide changes, and the push-based caching paradigm will outperform the pull-based paradigm. Moreover, we see there is an amount for $\beta$ where two paradigms have identical performance, resulting in zero percentage gain.

\begin{figure}[t]
\centering
\includegraphics[width=0.5\textwidth]{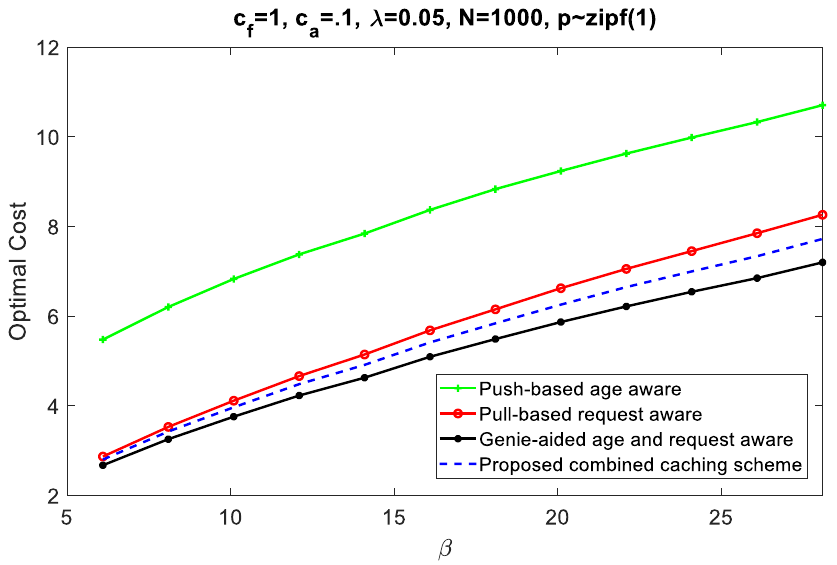}
\caption{Performance Comparison of the Proposed paradigms for Multi-item $N=1000$. }
\label{fig:comparison}
\end{figure}
The results shown in Fig.~\ref{fig:comparison} highlight the performance of the three caching paradigms in a multi-item scenario with $z=1$. The graph indicates that the pull-based caching paradigm, which only considers request information, outperforms the push-based caching paradigm for a range of $\beta$. Furthermore, the proposed combined caching strategy, which smartly switches between push- and pull-based caching for different items, delivers better results than either of the individual paradigms. This supports our expectation that the combined scheme would provide superior performance. However, as $\beta$ increases, the difference between the combined scheme and the individual paradigms increases. The graph also highlights the benefits of the combined scheme, as it takes advantage of the available information effectively, resulting in a cost close to that of the genie-aided paradigm that has complete information.

So far, we have proposed the combined caching policy assuming unlimited cache capacity. In reality, cache space limitations are often present. But given the dynamic nature of the content, these limitations are unlikely to be restrictive. As long as the cache space is sufficient for the dynamic content, the proposed combined policy can be applied without considering the cache capacity constraint. However, there may be scenarios where the cache constraint is active. In the next section, we extend the proposed combined caching policy to handle the case with a limited cache capacity, taking into account the buffer constraint.

\subsection{Caching with Buffer Constraint}
In this section, our focus is to minimize the cost stated in problem \eqref{eq:costmin}, given that the cache capacity is set to $B$. We assume that once an item is stored in the cache, it will remain there without being evicted. We then demonstrate how the proposed combined policy in Theorem \ref{thm:combinedparadigm} can be modified to be applied to the caching scenario with a buffer constraint.

According to Theorem \ref{thm:combinedparadigm}, for each data item $n \in \cN$, given its popularity and refresh rate, there exists a threshold on the fraction $\frac{p_n}{\lambda_n}$ where the push-based paradigm outperforms the pull-based paradigm. Furthermore, Theorem \ref{thm:combinedparadigm} showed that items can be sorted based on the decreasing order of the fraction $y_{n}^{*}=\frac{\beta p_n}{\lambda_n}$ and then divided into two separate groups based on the established threshold. The push-based caching paradigm is then applied to the first group and the pull-based caching paradigm to the second group. In this section, we demonstrate how such a policy can be generalized to the scenario of limited buffer size. The next proposition presents the optimal combined caching policy for the limited buffer case.

\begin{proposition}
\label{prop:bufferconstraint}
In a system composed of a data set $\cN$ of $N$ items with popularity distribution $\mathbf{p}=(p_n)_{n=1}^N$ and update rates $\boldsymbol{\lambda}=(\lambda_n)_{n=1}^N$ and a total cache capacity $B$, assume without loss of generality that items are ordered such that $y_{1}^{*} \geq y_{2}^{*} \geq ... \geq y_{N}^{*}$, where $y_{n}^{*}=\frac{\beta p_n}{\lambda_n}$. The optimal policy is given by:

\begin{equation}
\label{eq:bufferconstraint}
 \left\{\begin{array}{cc}
\mathcal{G}_1=\{1,2,...,n^*\}, \mathcal{G}_2=\{n^*+1,...,B\}  & n^* < B, \\
\hspace{-.7in} \mathcal{G}_1=\{1,2,...,B\}, \mathcal{G}_2=\{ \emptyset \}  & n^* \geq B,
\end{array}\right.
\end{equation}
where $n^*=\max \{ n | y_n^* >f^*\}$, $\mathcal{G}_1$ and $\mathcal{G}_2$ are the sets of items placed in the edge cache, and items in $\mathcal{G}_1$ are updated through a push-based paradigm, while items in $\mathcal{G}_2$ are updated through a pull-based paradigm. The average cost under this policy is given by:
\begin{equation} \label{eq:bufferconstraintcost}
\begin{split}
    C^*= & \sum_{n=1}^{ \min \{n^*,B\}  } \left( \sqrt{2 \lambda_n \beta p_n c_{a} c_{f}}-\frac{1}{2} \beta p_n c_{a}\right) \\
    + & \sum_{n^*+1}^{B}
        \lambda_n c_{a}\left(\sqrt{1+\frac{2 \beta p_n c_{f}}{\lambda_n c_{a}}}-1\right) 1\{ n^* < B \}.
\end{split}
\end{equation}
\end{proposition}

\begin{proof}
As we showed in Theorem \ref{thm:combinedparadigm}, the optimal policy without buffer constraint will divide the items into two disjoint groups based on the threshold obtained in Proposition \ref{prop:gain_threshold}. According to this proposition, there $\exists$ a threshold $f^*$ for $\frac{\beta p_n}{\lambda_n}$ that both push-based and pull-based policies will have identical performance, resulting in zero percentage gain. Moreover, if $\frac{\beta p_n}{\lambda_n}>f^*$, push-based caching performs better and if $\frac{\beta p_n}{\lambda_n}<f^*$, pull-based caching performs better than the other. Caching items with higher popularity and lower refresh rates will result in higher cost reductions. Since items are sorted according to $y_{n}^{*}=\frac{\beta p_n}{\lambda_n}$, therefore, the optimal policy will start filling the cache with the highest $y_{n}^{*}$ and will continue to do that until the cache is full. Thus, the set of cached items under the optimal policy is always $\{1,2,..., B\}$. After having the set of cached items and being equipped with knowledge of the optimal threshold $f^*$, we divide the set of cached items into two groups $\mathcal{G}_1$ and $\mathcal{G}_2$. The cached items of the group $\mathcal{G}_1$ are updated through a push-based paradigm, and the cached items of the group $\mathcal{G}_2$ are updated through a pull-based paradigm. Finally, according to Theorems \ref{thm:costpush} and \ref{thm:costpull}, the average cost can be expressed as in Equation \eqref{eq:bufferconstraintcost}.
\end{proof}

Proposition~\ref{prop:bufferconstraint} presents the proposed combined caching policy for handling the buffer constraint scenario. First, it fills the cache, and then, based on the optimal threshold $f^*$ for $\frac{\beta p_n}{\lambda_n}$, it determines whether to use the pull-based or push-based paradigm for the cached items.

To evaluate the performance of the proposed paradigms for the limited buffer size case, numerical simulations were performed and compared to the hypothetical genie-aided paradigm outlined in Proposition~\ref{prop:genie-aided}. The same parameters as in the previous section were used, and the buffer size was assumed to be $B=10$. We also assume that refresh rates $\lambda_n, \forall n$ are weighted according to Zipf with parameter $\alpha$ (i.e., $\lambda_n \propto \frac{1}{n^\alpha}$) where their average is considered to be $\lambda_{avg}=0.01.$

\begin{figure}[t]
\centering
\includegraphics[width=0.5\textwidth]{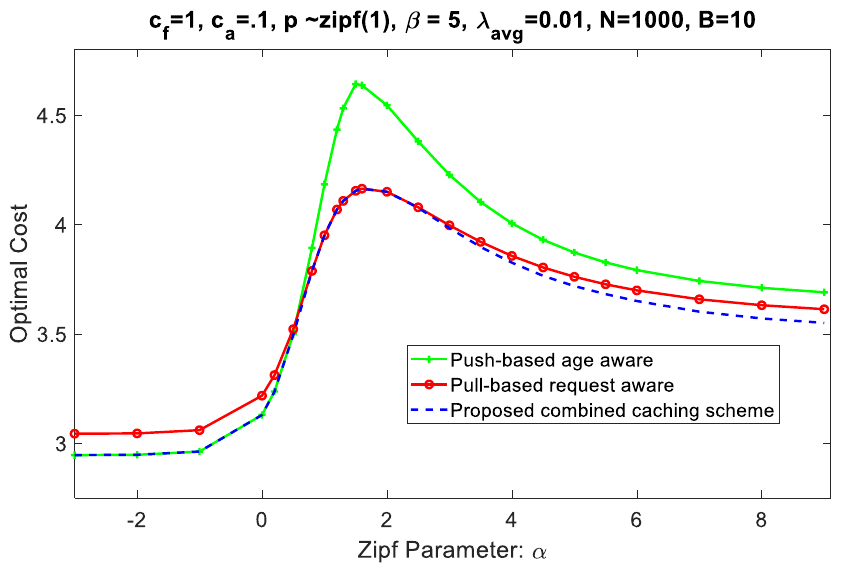}
\caption{Performance Comparison of the Proposed paradigms Under Limited Cache Capacity. }
\label{fig:comparison_limited_Zipf}
\end{figure}

Figure~\ref{fig:comparison_limited_Zipf} presents the performance of push- and pull-based caching paradigms with a limited cache capacity of $B=10$, as a function of the refresh Zipf parameter $\alpha$, for a Zipf(1) popularity distribution. When $\alpha>0$, highly popular items will have higher refresh rates, while for $\alpha<0$, items with higher popularity will have lower refresh rates. A value of $\alpha=0$ represents constant refresh rates for all items. As shown in the figure, when highly popular items have higher refresh rates (i.e., $\alpha > 0$), the pull-based caching paradigm, which is request-aware, outperforms the push-based paradigm, which is only age-aware. Conversely, when popular items have lower refresh rates, the push-based caching paradigm outperforms the pull-based paradigm. The proposed combined caching scheme, which dynamically switches between push- and pull-based approaches for different items, outperforms both paradigms at all levels of refresh rates represented by the Zipf parameter $\alpha$.

\begin{figure}[t]
\centering
\includegraphics[width=0.5\textwidth]{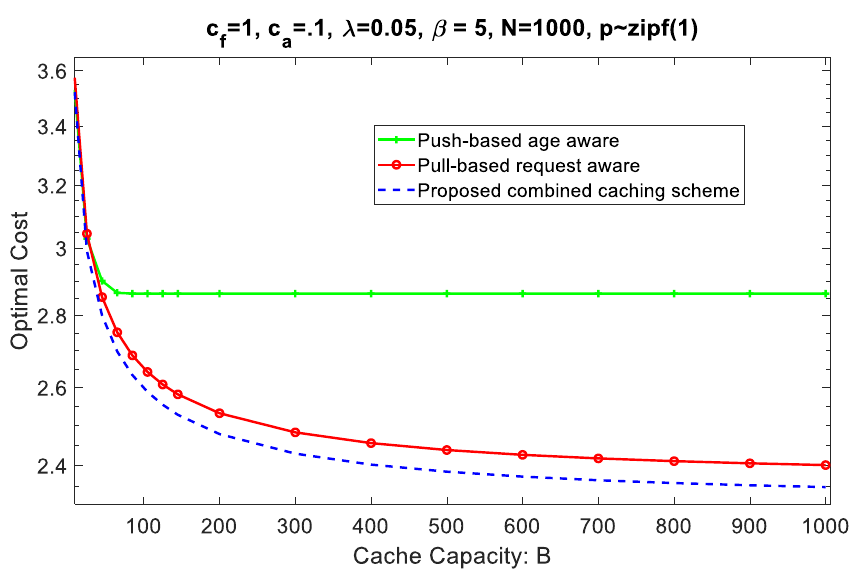}
\caption{Performance Comparison of the Proposed paradigms Under Limited Cache Capacity. }
\label{fig:comparison_limited_large}
\end{figure}

Figure~\ref{fig:comparison_limited_large} displays the performance of push- and pull-based caching paradigms as a function of increasing cache space, $B$. As the figure indicates, when the cache size is ample (e.g. $B>25$), pull-based caching surpasses push-based caching. It is also notable that push-based caching fails to take advantage of excess cache space, resulting in underutilization. On the other hand, pull-based caching consistently leverages additional cache space to minimize costs. However, as the figure highlights, the cost reduction from increasing cache space diminishes as the cache size becomes even larger (e.g. $B>N/10$). The proposed combined caching scheme not only outperforms both push- and pull-based approaches at all cache sizes, but also approaches optimal cost efficiency when the cache size is substantial (e.g. $B \approx 100$). This highlights the advantage of the proposed policy in effectively utilizing available cache space to achieve lower costs.

As demonstrated by Figures \ref{fig:comparison_limited_Zipf} and \ref{fig:comparison_limited_large}, the benefits of push-based caching are greatest when the cache size is limited and demand is more predictable. Conversely, pull-based caching is more effective in utilizing extra cache space and minimizing costs when the cache size is ample. The proposed combined caching policy is expected to outperform both push- and pull-based approaches in all scenarios

\section{Conclusion}
\label{sec:Conclusion}
In this work, we presented an important caching scenario for serving content that changes dynamically. We used the age-of-version metric to evaluate the freshness of served content and measure the number of stale versions per item. Our focus was on developing optimal caching strategies that minimize the system cost, which is influenced by the cost of fetching fresh content from a back-end database and the cost of aging content stored in a front-end cache. We examined the caching problem from two perspectives: the back-end database-driven caching paradigm, where the database controls cache updates, and the edge-cache-driven caching paradigm, where the local cache controls updates. We proposed a push-based caching policy for the former and characterized its optimality, while fully identifying the optimal caching policy for the latter. By comparing the performance of both paradigms, we determined the circumstances where each outperforms the other. Our new caching scheme, which combines the two strategies, outperforms either one individually and achieves the minimum cost among the discussed paradigms. Numerical results support the superiority of the proposed combined scheme over each of the individual paradigms.

\appendix
\subsection{Proof of Theorem 1:}
Since for each data item $n \in \cN$, the average cost is independent of all the other items, we can move the summation to the outside and write the optimization problem as:
\begin{equation} 
\begin{split}
    &\sum_{n=1}^N \min _{\boldsymbol{u_n}} \lim _{\mathrm{T} \rightarrow \infty} \frac{1}{T} \int_{o}^{\mathrm{T}}  \left[ u_n(t) c_f + \beta p_n  c_a  \Delta_n(t)\right]dt.
\end{split}
\end{equation}
Therefore, by decoupling the optimization over items, we can focus on optimizing the cost for each data item $n$, separately. We take the vanishing discounted cost approach to the average cost problem. Thus, consider:
\begin{equation} \label{eq:costpushvanish}
\begin{split}
    &\min _{\boldsymbol{u_n}} \lim _{\mathrm{T} \rightarrow \infty}  \mathbb{E}_{x}^{u} \int_{o}^{\mathrm{T}} e^{-\alpha t} C(\Delta_n(t),u_n(t))dt\\
    & s.t. \quad  u_n(t) \in \{0,1\}, \quad \forall t\geq 0,\\
    & \hspace{.32in}  \Delta_n(t+dt) =\Delta_n(t) (1-u_n(t))+b_n(t), \quad \forall t\geq 0.
\end{split}
\end{equation}
Next, based on the approach taken in \cite{rosberg1982optimal}, we characterize the equivalent discrete-time problem for the vanishing cost problem defined in \eqref{eq:costpushvanish}. For simplicity of notation, we use $C(\Delta_t)=C(\Delta_n(t),u_n(t))=u_n(t) c_f + \beta p_n  c_a  \Delta_n(t)$, where $\Delta_t=\Delta_n(t)$. Given a control policy $\boldsymbol{u}$ and an initial state $\Delta_0$, a semi-Markov decision process $(\Delta_t: t\geq0)$ with jump rates $\lambda_n$ is given by:
\begin{equation*}
    \Delta(t)=\Delta_{m(t)},
\end{equation*}
where $\boldsymbol{\Delta}=(\Delta_0,\Delta_1,...)$ is a sequence of random variables and $(\Delta(t):t\geq0)$ is a rate $\lambda_n$ Poisson process which is independent of $\boldsymbol{\Delta}$. Let $\tau_m$ denote the random times at which $(m(t))$ jumps for the $m$th time. Then, the cost for a control policy $\boldsymbol{u}$ over the random time interval $[0,\tau_m)$ is:
\begin{equation*} 
\begin{split}
   & \mathbb{E}_{x}^{u}\int_{o}^{\tau_m} e^{-\alpha t} C(\Delta_t)dt= \mathbb{E}_{x}^{u}\sum_{k=0}^{m-1} \int_{\tau_k}^{\tau_{k+1}}e^{-\alpha t}C(\Delta_t)dt\\
   =& \mathbb{E}_{x}^{u}\sum_{k=0}^{m-1} \left[C(\Delta_{t_k})\mathbb{E}\int_{\tau_k}^{\tau_{k+1}}e^{-\alpha t}dt\right]\\
   =&\mathbb{E}_{x}^{u}\sum_{k=0}^{m-1}C(\Delta_{t_k}) \frac{1}{\alpha} \mathbb{E} \left[e^{-\alpha t_k}-e^{-\alpha t_{k+1}}\right],
   \end{split}
\end{equation*}
where $t_k\leq \tau_m$ is the time of the $k$th jump. Since $t_{k+1}-t_k$'s are i.i.d. with $P\{t_{k+1}-t_k>t\}=e^{-t\lambda_n}$, therefore $\mathbb{E} e^{-\alpha t_k}=q^k$ where
\begin{equation*} 
    q=\mathbb{E} e^{-\alpha t_1}=\int_{0}^\infty e^{-\alpha t}\lambda_n e^{-\lambda_n t} dt=\frac{\lambda_n}{\alpha + \lambda_n}.
\end{equation*}
Substituting, we see that the cost equals:
\begin{equation*} 
 \mathbb{E}_{x}^{u}\int_{o}^{\tau_m} e^{-\alpha t} C(\Delta_t)dt=  \frac{1-q}{\alpha}  \mathbb{E}_{x}^{u} \sum_{k=0}^{m-1} q^k C(\Delta_{t_k}),
\end{equation*}
provided that $q<1$, whereas if $q=1$, then it equals $\frac{1}{\lambda_n}\mathbb{E}_{x}^{u} \sum_{k=0}^{m-1} q^k C(\Delta_{t_k})$. We will ignore the constant factor and take the cost to be:
\begin{equation*} 
 \mathbb{E}_{x}^{u} \sum_{k=0}^{m-1} q^k C(\Delta_{t_k}),
\end{equation*}
which is valid for $0\leq q\leq1$. Writing $\Delta_k=\Delta_{t_k},$ we see that the cost can also be viewed as the cost over $m$ time steps for a discrete-time decision process with discount factor $q$. Therefore, we focus on optimizing the discrete-time vanishing discounted cost problem stated as below: 
\begin{equation} \label{eq:pushdiscount} 
\begin{split}
    &\min _{\boldsymbol{u}} \lim _{m \rightarrow \infty}  \mathbb{E}_{x}^{u} \sum_{k=0}^{m-1} q^k C(\Delta_{t_k}),\\
\end{split}
\end{equation}
where $0<q<1$ is the discount factor. Since the cost function is non-negative, according to the value iteration algorithm, we define:
\begin{equation*} 
    v_{m}(x)=\min _{u_{n}}\left\{c(x, u)+q \sum_{y} v_{m-1}(y) \mathrm{p}(y \mid x, u)\right\},
\end{equation*}
for $\forall x, m \geq 1,$ where $x=\Delta$ is the state of the system and $u \in \{0,1\}$ is the action taken. Also, $v_0(x)=0, \forall x=\Delta \in \{0,1,2,...\}$. Such $\{ v_m \}_m$ is a monotonically non-decreasing sequence.
\begin{proposition} \label{prop:um}
There exists a threshold $\eta_m$ where the optimal action $u_m^*(\Delta)$ that minimizes $v_m(x), \forall n$ can be expressed as:
\begin{equation} \label{eq:hypothesis} 
u_{m}^{*}(\Delta)= \begin{cases}1 & \Delta>\eta_{m}, \\ 0 & \Delta \leq \eta_{m},\end{cases}
\end{equation}
\end{proposition}
where $\eta_m$ is the age threshold at step $m$.

\begin{proof}(Proposition \ref{prop:um})
For $m=1$, we have $\eta_1= \left\lfloor \frac{c_f}{\beta p_n c_a} \right\rfloor$. Thus, $v_1(\Delta)=\beta p_n c_a \Delta$  for $\Delta\leq \eta_1$ and $v_1(\Delta)=c_f$ for $\Delta > \eta_1$. By induction, assume that the hypothesis holds for $m-1$, we show that it also holds for $m$. We have:
\begin{equation*}
\begin{split}
    &v_{m}(\Delta)= \min   \{ \\& \underbrace{ \beta p_n c_a \Delta +q\left[v_{m-1}(\Delta+1)\right]}_{u_{m}=0}, \underbrace{c_f+q\left[v_{m-1}(1)\right]}_{u_{m}=1}\},
    \end{split}
\end{equation*}
where the first term is the projected cost when there is no cache update, i.e., $u_m=0$, and the second term is the projected cost of instances with cache update. Moreover, we know $v_{m-1}(\Delta)$ is a non-decreasing function of $\Delta$. Therefore, if for a given $\Delta$, $u_m^*(\Delta)=1$, then $\forall \Delta'>\Delta$ we have $u_m^*(\Delta')=1$. If for a given $\Delta$, $u_m^*(\Delta)=0$, then for $\forall \Delta'<\Delta$, we have $u_m^*(\Delta')=0$. Thus, there exists a threshold $\eta_m$ where Equation \eqref{eq:hypothesis} holds.\end{proof}

Based on Proposition \ref{prop:um} and according to the value iteration algorithm, the optimal policy that solves \eqref{eq:pushdiscount} has a threshold structure given in Equation \eqref{eq:hypothesis}. Next, we show that the thresholds $\eta_m, \forall m$ are uniformly bounded.
\begin{lemma}
The age thresholds $\eta_m, \forall m$ given in Equation \eqref{eq:pushdiscount} are uniformly bounded. In other words, there exists $\eta$ such that:
\begin{equation*}
    \eta_m\leq \eta , \quad \forall m
\end{equation*}
\end{lemma}
\begin{proof}
According to the value iteration algorithm:
\begin{equation*}
\begin{split}
    &v_{m}(\Delta)= \min   \{ \\& \underbrace{ \beta p_n c_a \Delta +q\left[v_{m-1}(\Delta+1)\right]}_{u_{m}=0}, \underbrace{c_f+q\left[v_{m-1}(1)\right]}_{u_{m}=1}\},
    \end{split}
\end{equation*}
since $v_{m-1}(\Delta), \forall m$ is a non-decreasing function of $\Delta$, so $q\left[v_{m-1}(\Delta+1)\right]\geq q\left[v_{m-1}(1)\right]$ holds for any $q \geq 0$ and $\Delta \geq 0$. Thus, independent of $q$, we have $\eta_m\leq \frac{c_f}{\beta p_n c_a}, \forall m.$ Consider $\eta=\left\lfloor \frac{c_f}{\beta p_n c_a} \right\rfloor+1$ and this completes the proof that the age thresholds are uniformly bounded.
\end{proof}

So far we showed that the policy that solves the vanishing discounted cost has a threshold structure. Then, based on the results of Lemma 1, we now argue that there exists a solution to the average cost minimization problem and this solution admits an optimal control of threshold type.

Consider the following controlled Markov chain with state space $\mathbb{X}=\{0,1,2,...,\eta,x_0\}$ and action space $\mathbb{U}=\{0,1\}$. Both state space and action space are finite. 

\begin{figure}[h]

\centering
\includegraphics[width=0.5\textwidth]{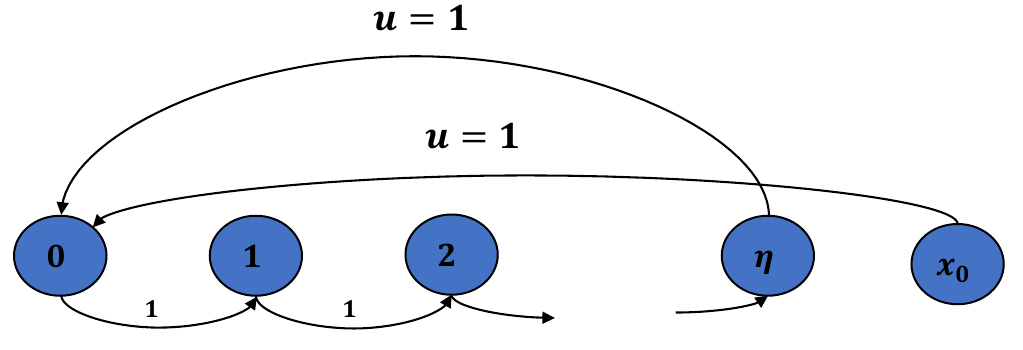}
\caption{ Markov chain diagram of age under the threshold policy with age threshold $\eta$ and initial state $x_0$.}
\label{fig:age MC}
\end{figure}

According to Proposition \ref{prop:um}, the optimal action $u_m^*(\Delta)$ that minimizes $v_m(x)$ has a threshold structure with age threshold $\eta_m <\eta$. Under such policy, the state space evolves similarly to what has been shown in Fig. \ref{fig:age MC}. The only difference is, since $\eta_m <\eta$, some states will not be visited. Therefore, at each step $m$, the state space would be a subset of the state space shown in Fig. \ref{fig:age MC} and thus it has a finite size. Since the state and action spaces are finite for any $q\in (0,1)$ and under the deterministic policy presented in Proposition \ref{prop:um}, the entire state space is a recurrent set as it is shown in Fig. \ref{fig:age MC}, then according to \cite[Theorem 7.3.1]{WinNT}, there exist a threshold policy which is optimal for the vanishing discounted cost when $q$ is sufficiently close to $1$, and such a policy is also optimal for the average cost.

And finally using the fact that the structure of the optimal policy is threshold based and utilizing the renewal process, we can find the optimal age threshold for each item $n$ by minimizing the cost function $C^{B}_n(m)$ over integer numbers. Substituting the optimal threshold in $C^{B}_n(m)$ and adding over all the items gives the optimal cost. And this completes the proof.

\subsection{Proof of Theorem 2:}
Similar to the proof of Theorem 1, since for each data item $n \in \cN$, the average cost is independent of all the other items, we can move the summation to the outside and write the optimization problem as:

\begin{equation} \label{eq:pullcontinous}
\begin{split}
    &\hspace{-.14in}  \sum_{n=1}^N \min _{\boldsymbol{u_n}} \lim _{\mathrm{T} \rightarrow \infty} \frac{1}{T} \int_{o}^{\mathrm{T}}  \left[u_n(t) c_f + r_n(t) s_n(t) \lambda_n c_a (1-u_n(t))\right]dt,
\end{split}
\end{equation}
Therefore, by decoupling the optimization over items, we can focus on optimizing the cost for each data item $n$, separately. More explicitly, we take the vanishing discounted cost approach to the average cost problem. Thus, consider:
\begin{equation} \label{eq:costpullvanish}
\begin{split}
    &\min _{\boldsymbol{u_n}} \lim _{\mathrm{T} \rightarrow \infty}  \mathbb{E}_{x}^{u} \int_{o}^{\mathrm{T}} e^{-\alpha t} C(x_n(t),u_n(t))dt\\
    & s.t. \quad  u_n(t) \in \{0,1\}, \quad \forall t\geq 0,\\
    & \hspace{.32in}  x_n(t)=(s_n(t),r_n(t)), r_n(t)\in \{0,1\}, s_n(t)\geq 0,\\
    & \hspace{.32in}  s_n(t+dt)=s_n(t)(1-u_n(t))+dt, \quad \forall n \in \cN , t\geq 0.
\end{split}
\end{equation}
In the next proposition, we show that by focusing  on the decision only at the jump times of $r$, the average cost cannot increase.

\begin{proposition} \label{prop:pulljump}
Assume $\boldsymbol{u_n}$ is an arbitrary policy. By changing the policy $\boldsymbol{u_n}$ such that each action $u_n(t)=1$ will wait until the next request, i.e., $r_n(t)=1$, we come up with a new policy $\boldsymbol{u'_n}$ which is a modified version of $\boldsymbol{u_n}$. Such policy $\boldsymbol{u'_n}$ will not perform worse than $\boldsymbol{u_n}$ for the average cost problem \eqref{eq:pullcontinous}.
\end{proposition}
\begin{proof} 
Consider policy $\boldsymbol{u}$, where there exists a time $t$ such that $r(t)=0$ and $u(t)=1$. In other words, $\boldsymbol{u}$ updates the cache when there is no request. Let $C^{\boldsymbol{u}}$ be the corresponding cost. Let $t'=\min \left\{\tau>t | r(\tau)=1 \right\}$ be the first arriving request after time $t$. We construct a new policy $\boldsymbol{u'}$ by setting $u'(\tau)=0$ for $t \leq \tau<t'$ and $u'(t')=1$. Now, we show that this modification does not increase the average cost of Equation \eqref{eq:pullcontinous}. Since the two policies are identical outside the interval $[t,t']$, we focus on this time period. Policy $\boldsymbol{u}$ updates the cache at time $t$ with cost $c_f$. When the next request arrives at $t'$, the cached content might be old, incurring an average freshness cost of $c_a\lambda_n (t'-t)$. Thus, $C^{\boldsymbol{u}}([t,t'])=c_f+c_a\lambda_n (t'-t)$. In contrast, policy $\boldsymbol{u'}$ updates the cache when there is a request, eliminating potential freshness cost. Thus, $C^{\boldsymbol{u'}}([t,t'])=c_f$. Therefore, by this slight modification to $\boldsymbol{u}$, the cost is reduced.
\end{proof}

If we keep applying this policy to every update instance of policy $\boldsymbol{u}$, ensuring cache updates only happen at times of request arrivals, the result is a policy that updates the cache only at time instances when there is a request. And such a policy will not perform worse than the original policy $\boldsymbol{u}$. Therefore, according to proposition \ref{prop:pulljump} and by taking the same approach of Appendix A, we see that the cost can also be viewed as the cost over $m$ time steps for a discrete-time decision process with discount factor $q$ where the time steps are the jump times of $r_n$ which is a rate $\beta p_n$ Poisson process. Thus, we focus on optimizing the discrete-time vanishing discounted cost problem stated below: 
\begin{equation} \label{eq:pulldiscount} 
\begin{split}
    &\min _{\boldsymbol{u}} \lim _{m \rightarrow \infty}  \mathbb{E}_{x}^{u} \sum_{k=0}^{m-1} q^k C(x_{t_k}),\\
\end{split}
\end{equation}
where $0<q<1$ is the discount factor and $x_{t_k}=(s_{t_k},r_{t_k})$ is the state at time step 
$k$. Since the cost function is non-negative, according to the value iteration algorithm, we define:
\begin{equation*} 
    v_{m}(x)=\min _{u_{n}}\left\{c(x, u)+q \sum_{y} v_{m-1}(y) \mathrm{p}(y \mid x, u)\right\}, \forall x, m,
\end{equation*}
where $x=(s,r)$ is the state of the system and $u \in \{0,1\}$ is the action taken. Also, $v_0(x)=0$ for all $x$. Such $v_m$ is a monotonically non-decreasing sequence.
\begin{lemma}
$v_{m}(s, r), \forall m \geq 1, r \in\{0,1\}$ is a non-decreasing function of $s$.
\end{lemma}
\begin{proof}(Lemma 1): We use induction. According to the definition of $v_m(s,r)$, we can show that:
\begin{equation*} 
    v_{1}(s, r)= \begin{cases}r c_{f} & s>c_{f} / \lambda_n c_{a}, \\ \lambda_n c_{a} r s & s \leq c_{f} / \lambda_n c_{a},\end{cases}
\end{equation*}
which is a non-decreasing function of $s$. Assume $v_{m-1}(s,r)$ is non-decreasing over $s$, then we have:
\begin{equation*} 
\begin{split}
    v_{m}(s, r)=\min & \{\lambda_n c_{a} s r+q\beta p_n v_{m-1}(s+1,1), \\&
    c_{f}+q\beta p_n v_{m-1}(1,1)\}.
    \end{split}
\end{equation*}
Therefore, $v_{m}(s, r), \forall m \geq 1, r \in\{0,1\}$ which is the minimum of two non-decreasing functions of $s$, is itself a non-decreasing function of $s$.\end{proof}

\begin{proposition} \label{prop:pulloptimalr}
For $r=1$, there exists a $\tau_m$ such that the policy that minimizes $v_m(s,1)$ is given by:
\begin{equation} 
\label{eq:u(s,1)}
    u_{m}^{*}(s, 1)= \begin{cases}1 & s>\tau_{m} \\ 0 & s \leq \tau_{m}\end{cases}.
\end{equation}
\end{proposition}
\begin{proof}(Proposition \ref{prop:pulloptimalr}): For $m=1$, we have $\tau_{1}=c_{f} / \lambda_n c_{a}$ and the assumption holds. According to the definition of $v_m(s,r)$ and focusing on the jump times of $r_n$, we have:
\begin{equation*} 
\begin{split}
    \hspace{-.1in} v_{m}(s, 1)=\min  \{ \underbrace{  \lambda_n c_{a} s +q v_{m-1}(s+1,1)}_{u_{m}=0},  \underbrace{c_{f}+q v_{m-1}(1,1)}_{u_{n}=1}\},
    \end{split}
\end{equation*}
where the first term is the projected cost when the request is served from the edge-cache, i.e., $u_m=0$, and the second term is the projected cost of fetching the request from the back-end database. Moreover, according to Lemma 1, $v_{m-1}(s,r)$ is a non-decreasing function of $s$. Therefore, if for a given $s$, $u_m^*(s,r)=1$, then for $\forall s'>s$ we have $u_m^*(s',r)=1$. Besides, if for a given $s$, $u_m^*(s,r)=0$, then for $\forall s'>s$, we have $u_m^*(s',r)=0$. Thus, there exists a threshold $\tau_m$ where the Equation \eqref{eq:u(s,1)} holds.\end{proof}

According to Propositions \ref{prop:pulljump} and \ref{prop:pulloptimalr}, there exists a threshold $\tau_m$ where the optimal action $u_m^*(s,r)$ that minimizes $v_m(x), \forall m$ can be expressed as:
\begin{equation} \label{eq:pullhypothesis} 
u_{m}^{*}(s, r)= \begin{cases}r & s>\tau_{m} \\ 0 & s \leq \tau_{m}\end{cases}.
\end{equation}

Next, we show that the sequence $\{v_m(x)\}$ converges to $v(x)$ as $m$ grows. Since $\{v_m\}$ is a monotonically non-decreasing sequence, we show that there exists a policy that when $m$ increases, the cost is always finite. Assume that $\tau$ is given and is constant. Consider the policy $\bar{u}_{\tau}=\bar{u}_{2\tau}=\bar{u}_{3\tau}=...=1$, where we only fetch every $\tau$ time slot. Under such policy we have $(1-q^\tau)v_{\infty}^{\bar{u}}=g(q)<\infty$ where $g(q)=\beta \lambda \sum_{n=1}^{\tau-1}nq^n+q^\tau c_f$ is always finite for a given $\tau$ and $q<1$. Therefore, $v_m(x)\leq v_{m+1}(x)\leq ... \leq v_{\infty}^{\bar{u}}$. According to monotone convergence theorem, there exist $v(x)$ such that $v_{m}(x) \stackrel{m \rightarrow \infty}{\longrightarrow} v(x)$. Thus, the optimal policy that minimizes $v(x)$ is obtained through $\lim_{m\to \infty } u_{m}^{*}(s, r) =u^{*}(s, r)$ and is given by:
\begin{equation*} 
u^{*}(s, r)= \begin{cases}r & s>\tau^*\\ 0 & s \leq \tau^* \end{cases},
\end{equation*}
where $\tau^*= \lim_{m\to \infty } \tau_m \leq \frac{c_f}{\lambda c_a}$, since $\tau_1=\frac{c_f}{\lambda c_a}, \forall q \in (0,1)$ and $0 \leq \tau_{m}\leq \tau _1, \forall m.$ The later follows from Lemma 1 that for any given $q \in (0,1)$ we have $v_{m-1}(s+1,1)\geq v_{m-1}(1,1), \forall m, s\geq 0.$ Then, for the average cost setup, we now argue that there exists a solution to the average cost minimization problem and this solution admits an optimal control of threshold type. 

The existence of a solution to the average cost optimality follows from the moment nature of the cost function and the existence of a solution leading to a finite cost, via the convex analytic method \cite{Borkar2} and its implication for average cost optimality inequality under a stable control policy via \cite[Prop. 5.2]{hernandez1993existence}. See \cite[Theorem 2.1 and Section 3.2.1]{arapostathis2021optimality} for further discussions for both existence and optimality inequality, respectively.

Now, consider the relative value functions solving the discounted cost optimality equation for each $q \in (0,1):
J_q(x) - J_q(x_0) = \min_u (c(x,u) + q \int (J_q(y) - J_q(x_0))p(dy|x,u) )$
Via the threshold optimality of finite horizon problems, since the threshold policies, $\tau_q$, are uniformly bounded, there exists a converging subsequence so that $\tau_{q_m} \to \tau^*$. Then 
$\lim_{m \to \infty} J_{q_m}(x) - J_{q_m}(x_0) =\lim_{m \to \infty}  \min_u (c(x,u) + q_m \int (J_{q_m}(y) - J_{q_m}(x_0))p(dy|x,u) )
= \lim_{m \to \infty} (c(x,u_{\tau_{q_m}}) + q_m \int (J_{q_m}(y) - J_{q_m}(x_0))p(dy|x,u_{\tau_{q_m}}))$
By considering a converging sequence and from \cite[Theorem 5.4.3]{WinNT}, we have that average cost optimality inequality applies under the limit threshold policy. Thus by relating the DCOE and ACOI, the optimality of threshold policies follows.

Finally, using the obtained result that the optimal policy for the average cost problem has a threshold structure, the optimal threshold, and the subsequent optimal cost are calculated in \cite{abolhassani2021fresh} (Theorem 1) for a threshold-based policy.

\bibliographystyle{IEEEtran}
\bibliography{paper}

\begin{IEEEbiography}[{\includegraphics[width=1.5in,height=1.3in,clip,keepaspectratio]{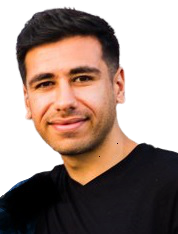}}\vspace{-.1in}]{Bahman Abolhassani } is a Presidential Postdoctoral Fellow at the Bradley Department of Electrical and Computer Engineering at Virginia Tech, Blacksburg, VA. He received his Ph.D. degree in electrical engineering from the ECE Department at The Ohio
State University in 2023. He also received the B.Sc. and M.Sc. degrees in electrical engineering from the EE Department at Sharif University of Technology in 2015 and 2017, respectively. His research interests include communication networks, optimization, machine learning, and cybersecurity.

\end{IEEEbiography}

\begin{IEEEbiography}
    [{\includegraphics[width=1.1in,height=1.4in,clip,keepaspectratio]{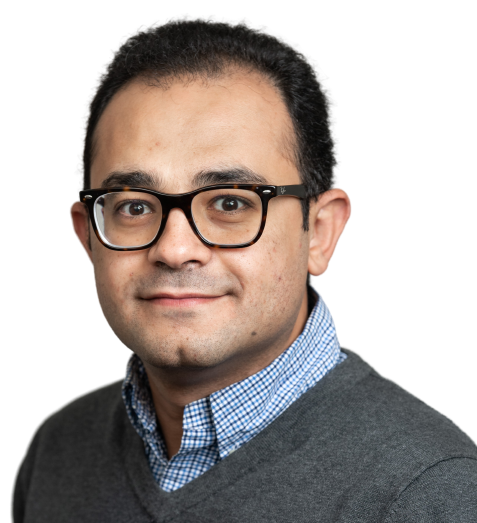}}\vspace{-.2in}]{John Tadrous }is an associate professor of electrical and computer engineering at Gonzaga University. He received his Ph.D. degree in electrical engineering from the ECE Department at The Ohio State University in 2014, MSc degree in wireless communications from the Center of Information Technology at Nile University in 2010, and BSc degree from the EE Department at Cairo University in 2008. Between 2016 and 2021 he served as an assistant professor of electrical and computer engineering at Gonzaga University. From May 2014 to August 2016, he was a post-doctoral research associate with the ECE Department at Rice University. His research interests include machine learning, autonomous robots, network optimization, and cybersecurity. Dr. Tadrous served as a technical program committee member for several conferences such as Mobihoc, COMSNETS, and WiOpt. 
\end{IEEEbiography}

\begin{IEEEbiography}
    [{\includegraphics[width=1.1in,height=1.4in,clip,keepaspectratio]{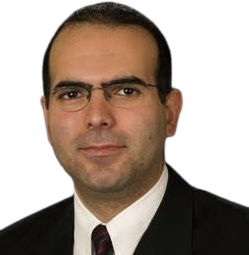}}\vspace{-.2in}]{Atilla Eryilmaz} (S'00 / M'06 / SM'17) received his M.S. and Ph.D. degrees in Electrical and Computer Engineering from the University of Illinois at Urbana-Champaign in 2001 and 2005, respectively. Between 2005 and 2007, he worked as a Postdoctoral Associate at the Laboratory for Information and Decision Systems at the Massachusetts Institute of Technology. Since 2007, he has been at The Ohio State University, where he is currently a Professor and the Graduate Studies Chair of the Electrical and Computer Engineering Department. Dr. Eryilmaz's research interests span optimal control of stochastic networks, machine learning, optimization, and information theory. He received the NSF-CAREER Award in 2010 and two Lumley Research Awards for Research Excellence in 2010 and 2015. He is a co-author of the 2012 IEEE WiOpt Conference Best Student Paper, subsequently received the 2016 IEEE Infocom, 2017 IEEE WiOpt, 2018 IEEE WiOpt, and 2019 IEEE Infocom Best Paper Awards. He has served as: a TPC co-chair of IEEE WiOpt in 2014, ACM Mobihoc in 2017, and IEEE Infocom in 2022; an Associate Editor (AE) of IEEE/ACM Transactions on Networking between 2015 and 2019; an AE of IEEE Transactions on Network Science and Engineering between 2017-2022; and is currently an AE of the IEEE Transactions on Information Theory since 2022.
\end{IEEEbiography}

\begin{IEEEbiography}
[{\includegraphics[width=1.1in,height=1.4in,clip,keepaspectratio]{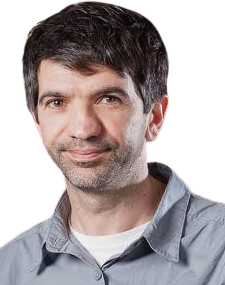}}\vspace{-.05in}]{Serdar Yüksel}(S'02, M'11) received his B.Sc. degree in Electrical and Electronics Engineering from Bilkent University in 2001; M.S. and Ph.D. degrees in Electrical and Computer Engineering from the University of Illinois at Urbana-
Champaign in 2003 and 2006, respectively. He was a post-doctoral researcher at Yale University before joining Queen's University as an Assistant Professor in the Department of Mathematics and Statistics, where he is now a Professor. His research interests are on stochastic control theory, information theory and probability. Prof. Yüksel is a co-author of several books, recipient of several awards; and is a senior editor with IEEE Transactions on Automatic Control and has been an Associate Editor for IEEE Transactions on Automatic Control, Automatica, Systems and Control Letters, and Mathematics of Control, Signals, and Systems. 
\end{IEEEbiography}

\end{document}